\documentclass[11pt,a4paper,reqno]{amsart} 
\usepackage[utf8]{inputenc}

\usepackage{lmodern}

\usepackage{a4wide}
\usepackage{fancyhdr}
\usepackage[english]{babel}

\usepackage{amsmath, mathtools}
\usepackage{amsfonts}
\usepackage{amssymb}
\usepackage{amsthm}
\usepackage[dvipsnames]{xcolor}
\usepackage{tikz}
\usepackage{pgfplots}
\usetikzlibrary{positioning}
\usetikzlibrary{intersections}
\usetikzlibrary{calc}
\usepackage{csquotes}
\usepackage{mathabx}
\usepackage{bm}
\usepackage{todonotes}
\usepackage{hyperref}
\usepackage{cleveref}
\usepackage{bbm}
\usepackage{graphicx}
\usepackage{subcaption}

\theoremstyle{definition}

\numberwithin{equation}{section}

\title[Stationary Soap Film Bridges Formed by a Small Electrostatic Force]{Stationary Soap Film Bridge Formed by a Small Electrostatic Force}
\author{Lina Sophie Schmitz}
\date{September 27, 2024}
\address{Institut f\"ur Angewandte Mathematik, Leibniz Universit\"at Hannover, Welfengarten 1 \\ D-30167 Hannover, Germany}
\email{schmitz@ifam.uni-hannover.de}

\usepackage{graphicx}

\newtheoremstyle{common}
	{}    
	{}    
    {\itshape}
    {0em}
    {\bfseries}
    {}
    {.5em}
    {}
\theoremstyle{common}

\numberwithin{subsection}{section}

\numberwithin{figure}{section}

\newtheorem{thm}{{\bf Theorem}}
\numberwithin{thm}{section}

\newtheorem{lem}[thm]{{\bf Lemma}}
\newtheorem{cor}[thm]{{\bf Corollary}}
\newtheorem{prop}[thm]{{\bf Proposition}}

\numberwithin{equation}{section}

\newtheoremstyle{commondef}  
    {6pt}
    {6pt}
    {}
    {0em}
    {\bfseries}
    {}
    {.5em}
    {}
\theoremstyle{commondef}

\newtheorem{bem}[thm]{{\bf Remark}}
\newtheorem{bems}[thm]{{\bf Remarks}}

\renewenvironment{proof}{{\bf Proof}.}{\qed\\}

\begin{document}

\begin{abstract} We consider two models, a free boundary problem and a simplification thereof, which describe a soap film bridge subjected to an electrostatic force. For both models, we construct stationary solutions if the force is small, analyse their stability and examine how their shape is influenced by small changes in the strength of the force. 
\end{abstract}

\maketitle

\allowdisplaybreaks

\section{Introduction}

We study a tiny soap film bridge spanned between to parallel rings and placed inside a metal cylinder \cite{MP09}. A voltage is applied between the cylinder and the soap film which induces an electrostatic force pulling the film outwards, see Figure \ref{fig1}. In the following, we consider the problem for small voltages, and ask, in particular, how the film responds to an increase of the electrostatic force. We present rigorous answers within the framework of two models:

\subsection{Free Boundary Problem}

The first model is the stationary version of \cite{LSS24a}. It  reads
\begin{subequations}\label{eq11}
 \begin{align}
 \begin{cases}
 -\sigma\, \partial_z \mathrm{arctan}  ( \sigma\partial_z u )&=  -\displaystyle\frac{1}{u+1} +\lambda g(u) \\
  \qquad \qquad \qquad u(\pm 1 )&=0 \,, \qquad -1 < u <1\,,  
 \end{cases} \label{eq1}
\end{align}
with electrostatic force
\begin{align}
g(u):= (1+\sigma^2(\partial_z u)^2)^{3/2} \vert \partial_r \psi_u(z,u+1) \vert^2\,,\label{eq2.5}
\end{align}
where 
\begin{align}
\begin{cases}\displaystyle
\frac{1}{r} \partial_r \left ( r \partial_r \psi_u \right ) + \sigma^2 \partial^2_z \psi_u&= 0  \quad \ \text{in} \quad \Omega(u)\,,  \\
 \ \, \qquad \qquad \qquad \qquad \psi_u &= h_u \quad \text{on} \quad \partial \Omega(u) \,,
\end{cases} \label{eq2}
\end{align}
and
\begin{align}
h_u(z,r) := \frac{\ln \Big (\displaystyle\frac{r}{u(z)+1} \Big )}{\ln \Big (\displaystyle\frac{2}{u(z)+1} \Big )}\,. \label{eq3}
\end{align} 
\end{subequations}
 Herein, $u+1$ with $u=u(z): (-1,1) \rightarrow (-1,1)$ gives the profile of the soap film bridge, 
$\Omega(u) = \big \lbrace (z,r) \in (-1,1) \times (0,2) \, \vert \, u(z)+1 < r < 2 \, \big \rbrace$ is the domain between cylinder and film, and
$\psi_u=\psi_u(z,r): \overline{\Omega(u)} \rightarrow \mathbb{R}$ is the electrostatic potential. The subproblem \eqref{eq2} is always considered in dependence on $u$, and the boundary condition \eqref{eq3} results from neglecting the fringing field. The parameter $\sigma$ gives the ratio of radii of the rings divided by their distance, and $\lambda \in [0,\infty)$ gives the strength of the applied voltage. The problem \eqref{eq11} is related to the class of models in \cite{LW17}, in particular to  \cite{ELW14,ELW15} therein.

 \begin{figure}[h]
\begin{subfigure}[l]{0.4\textwidth}
\begin{tikzpicture}
\begin{axis}[samples=200, axis x line=none, axis y line=none, domain=-1:1, clip=false]
\addplot[color=blue, thick]({0.7*x},{-1+0.17*(1-x^2)^0.5});
\addplot[color=blue, thick]({0.7*x},{1+0.17*(1-x^2)^0.5});
\addplot[color=purple, thick]({1.1*x},{1+0.3*(1-x^2)^0.5});
\addplot[color=purple!40, thick]({1.1*x},{-1+0.3*(1-x^2)^0.5});
\addplot[->, domain=-1.2:1.5]({0},{x});
\addplot[domain=-0.02:0.02]({x},{-1});
\addplot[domain=-0.02:0.02]({x},{1});

\addplot[color=purple, thick]({-1.1},{x});
\addplot[color=purple, thick]({1.1},{x});
\addplot[->, domain=-1.2:1.5]({0},{x});

\addplot[color=blue, thick]({cosh(x)/cosh(1)-0.3},{x});
\addplot[color=blue, thick]({-cosh(x)/cosh(1)+0.3},{x});

\addplot[color=blue,thick]({0.7*x},{(1-0.17*(1-x^2)^0.5)});
\addplot[color=blue, thick]({0.7*x},{-1-0.17*(1-x^2)^0.5});

\addplot[color=purple,thick]({1.1*x},{(1-0.3*(1-x^2)^0.5)});
\addplot[color=purple, thick]({1.1*x},{-1-0.3*(1-x^2)^0.5});

\end{axis}

\draw [->, bend angle=45, bend right]  (-0.2,4) to (0.4,3.6);
\node[text width=3.5cm] at (0,4.7) {\begin{small} metal cylinder \\ held at positive\\ potential \end{small}};

\draw [->, bend angle=30, bend left]  (0,2.7) to (2.3,2);
\node[text width=3.5cm] at (0,2.7) {\begin{small} soap film \\ \ held at \\ \ potential $0$ \end{small}};

\end{tikzpicture}
\caption{ Soap film bridge placed inside a metal cylinder.} 
\end{subfigure}
\hfill
\begin{subfigure}[r]{0.4\textwidth}
\vspace{1cm}
\begin{tikzpicture}
\fill[white,rounded corners] (-3.5,-0.8) rectangle (3.3,3.5);
\draw[domain=-2.45:2.45, smooth, variable=\z, blue, thick] plot ({\z},{cosh(0.5*\z)-0.3});
\draw[purple, thick] (-2.45,3) -- (2.45,3) ;
\draw[dashed](2.45,0) -- (2.45,3) ;
\draw[dashed] (-2.45,3) -- (-2.45,0) ;
\draw[dashed] (-2.45,3) -- (-2.45,0) ;
\draw[dashed] (-2.45,3) -- (-2.45,0) ;
\draw[->] (-3,0) -- (3,0) node[right] {$z$};
\draw (2.45,0.1) -- (2.45,-0.1) node[below] {$1$};
\draw (-2.45,0.1) -- (-2.45,-0.1) node[below] {$-1$};
\draw[->] (-2.9,-0.1) -- (-2.9,3.2) node[right] {$r$};
\draw (-2.8,1.5) -- (-3,1.5) node[left] {$1$};
\draw (-2.8,3) -- (-3,3) node[left] {$2$};
\node[above] at (0,1.7) {$\Omega(u)$};
\draw[->, blue, thick] (1.4,0) -- (1.4,0.9);
\node[blue, left] at (1.3,0.4) {$u+1$};
\draw [<-, bend angle=45, bend left]  (-0.8,3) to (-1,2.6) node[left]{$\psi_u=1$};
\draw [<-, bend angle=45, bend left]  (0.8,0.86) to (1,1.7) node[right]{$\psi_u=0$};
\end{tikzpicture}
\caption{Cross section of the soap film bridge with profile $u+1$ inside a metal cylinder.}
\end{subfigure}
\caption{Depiction of the problem set-up and its cross section.}\label{Cross Section}\label{fig1}
\end{figure}
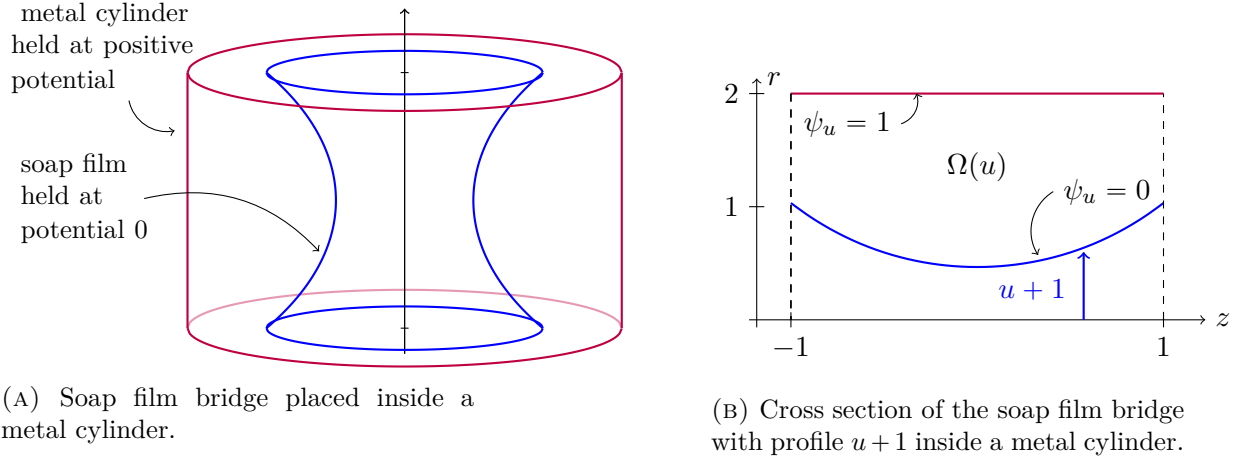

\subsection{Small Aspect Ratio Model}\label{Simplified Model} The second model
\begin{subequations}\label{A.56}
\begin{align}
 \begin{cases}\displaystyle
  \quad\   -\sigma\, \partial_z \mathrm{arctan}  ( \sigma\partial_z u )&=  -\displaystyle\frac{1}{u+1} +\lambda\,g_{sar}(u)\,, \\
  \qquad \qquad \qquad u( \pm 1 )&=0 \,, \qquad -1 < u <1\,, 
 \end{cases} \label{A.5}
\end{align}
with explicitly given electrostatic force 
\begin{align}
g_{sar} (u):= (1+\sigma^2(\partial_z u)^2 )^{1/2}\,\frac{1}{(u+1)^2 \ln^2 \Big (\displaystyle \frac{2}{u+1} \Big )} \label{A.6}
\end{align}
\end{subequations}
has been introduced in \cite{MP09}. The subscript $sar$ in $g_{sar}$ stands for small aspect ratio, as \eqref{A.56} is derived under the assumption that the gap between the rings on which the soap film is spanned and the cylinder compared to the distance of the rings is small. In contrast to \eqref{eq11}, the small aspect ratio model \eqref{A.56} consists of a singular ordinary differential equation in which $\psi_u$ is eliminated. For a derivation of \eqref{A.56} from (a slightly more general variant) of \eqref{eq11}, see \cite[Appendix B]{LSS24}. \\

Note that for the dynamical version of \eqref{eq11} and \eqref{A.56} the first lines in \eqref{eq1} and \eqref{A.5} have to be replaced by
\begin{align}
\partial_t u - \sigma \partial_z \mathrm{arctan} (\sigma \partial_z u) = - \frac{1}{u+1} + \lambda g_*(u) \label{dyn}
\end{align} 
with $g_*=g$ or $g_*=g_{sar}$ respectively. Also, an initial value $u_0$ is required.

\subsection{Catenoids as Stationary Solution} 
Our investigation starts with stationary solutions in absence of an electrostatic force, i.e. for $\lambda=0$. In this case, \eqref{eq11} and \eqref{A.56} coincide and become a minimal surface equation
\begin{align}
\begin{cases}
  &-\sigma \partial_z \mathrm{arctan} (\sigma \partial_z u) = - \displaystyle\frac{1}{u+1}\,,\\
  &u(\pm1)=0 \,, \quad -1 <u < 1\,.
\end{cases} \label{statlambda10}
\end{align}
It is well-known, see \cite[p.\,282]{JLJ98}, that there exists $\sigma_{crit}>0$ such that \eqref{statlambda10} has:
\begin{itemize}
 \item no solution for $\sigma < \sigma_{crit}$\,,
 \item exactly one solution for $\sigma=\sigma_{crit}$\,,
 \item exactly two solutions for $\sigma > \sigma_{crit}$\,.
 \end{itemize}
The critical value is $\sigma_{crit}= {\frac{\mathrm{cosh}(c_{crit})}{c_{crit}}} \approx 1.5$ with $c_{crit} \approx 1.2$ being the solution to
\begin{align}
c_{crit} \mathrm{sinh}(c_{crit})-\mathrm{cosh}(c_{crit}) =0 \,. \label{Defccrit}
\end{align}
Each solution to \eqref{statlambda10} is a (translated) catenoid 
\begin{align}
u_*(z) := \frac{\mathrm{cosh}(cz)}{\mathrm{cosh}(c)\ } - 1 \,, \qquad z \in (-1,1)\,, \label{catenoid}
\end{align} 
where $c>0$ satisfies 
\begin{align}\sigma = \frac{\mathrm{cosh}(c)}{c}\,. \label{261023b} 
\end{align}
For $\sigma > \sigma_{crit}$, there are two solutions $c=c_{in}$ and $c=c_{out}$ to \eqref{261023b} with
\begin{align}
c_{out} < c_{crit} < c_{in}\,, \label{crel}
\end{align}
resulting in an inner catenoid $u_{in}$ for $c_{in}$ and an outer catenoid $u_{out}$ for $c_{out}$ with $u_{out} > u_{in}$ in $(-1,1)$.

\subsection{Main Results}

Concerning the free boundary problem \eqref{eq11} we prove three main results. First, we show the existence of at least two stationary solutions for small $\lambda >0$ and $\sigma > \sigma_{crit}$:

\begin{thm}\label{exstatsL}{\bf(Existence)}\\
Let $q \in (2,\infty)$ and $\sigma > \sigma_{crit}$. Then, there exists $\delta =\delta(\sigma)> 0$ and analytic functions 
\begin{align*}
[\lambda \mapsto u^{\lambda}_{in}] &: [0,\delta) \rightarrow W^2_{q,D}(-1,1) \,, \ \qquad u_{in}^0=u_{in}\,, \\
[\lambda \mapsto u^{\lambda}_{out}] &: [0,\delta) \rightarrow W^2_{q,D}(-1,1) \,, \qquad u_{out}^0=u_{out}
\end{align*}
such that $u^\lambda_{in}$ and $u^\lambda_{out}$ are two different solutions to \eqref{eq11} for each $\lambda \in (0,\delta)$. Moreover, $u^\lambda_{in}$ and $u^\lambda_{out}$ as well as the corresponding electrostatic
potentials $\psi_{u^\lambda_{in}} \in W^2_2\big (\Omega(u^\lambda_{in})\big)$ and $\psi_{u^\lambda_{out}} \in W^2_2\big (\Omega(u^\lambda_{out})\big)$ are symmetric with respect to the $r$-axis.
\end{thm}

Here, $W^2_{q,D}(-1,1)$ consists of Sobolev functions with zero trace. Theorem \ref{exstatsL} follows from the implicit function theorem and the proof is contained in Section \ref{SectionStationary1}. We refer to \cite{ELW14, ELW15} for previous results on related models.  As a second result, we provide details on stability of stationary solutions to \eqref{eq11} under rotationally invariant perturbations in the presence of a small voltage:

\begin{thm}{\bf(Stability)}\label{stability}\\
 Let $q \in (2,\infty)$ and $\sigma > \sigma_{crit}$. Then, there exists $\delta=\delta(\sigma)> 0$ such that for each $\lambda \in [0,\delta)$: \\
 \vspace{1mm}
 {\bf(i)}\,The stationary solution $u^\lambda_{in}$ to \eqref{eq11} is unstable in $W^2_{q,D}(-1,1)$.\\
 \vspace{1mm}
 {\bf(ii)}\,The stationary solution $u^\lambda_{out}$ to \eqref{eq11} is exponentially asymptotically stable in $W^2_{q,D}(-1,1)$. More precisely, 
 there exist numbers $\omega_0, m, M > 0$ such that for each initial value $u_0 \in W^2_{q,D}(-1,1)$ with $\Vert u_0 - u^\lambda_{out} \Vert_{W^2_{q,D}} < m$, the solution $u$ to the dynamical version of \eqref{eq11}, see \eqref{dyn}, exists globally in time and the estimate
 \begin{align*}
  \Vert u(t)-u^\lambda_{out} \Vert_{W^2_{q,D}(-1,1)} &+ \Vert \partial_t u(t) \Vert_{L_q(-1,1)} \leq M \, e^{-\omega_0 t} \Vert u_0 -u^\lambda_{out} \Vert_{W^2_{q,D}(-1,1)} 
 \end{align*}
 holds for $t \geq 0$.
\end{thm}

We prove Theorem \ref{stability} in Section \ref{SectionStationary2}, where we roughly follow \cite{ELW14,ELW15} and apply the principle of linearized stability. Next, we show that the stable stationary solutions $u^\lambda_{out}$ stemming from $u_{out}$ are deflected outwards for small $\lambda$. 

\begin{thm}{\bf (Direction of Deflection in \eqref{eq11})}\label{DeflectionOut} \\
For fixed $\sigma>\sigma_{crit}$, there exists $\delta >0$ such that 
\begin{align*}
u_{out}^{\overline{\lambda}} (z) < u_{out}^{\lambda}(z) \,, \qquad 0 \leq \overline{\lambda} < \lambda < \delta \,, \quad z \in (-1,1)\,.
\end{align*}
\end{thm}
Theorem \ref{DeflectionOut} reflects a physically expected behaviour: A larger electrostatic force pulls stable configurations of the film outwards. The proof relies on a functional analytic version of the maximum principle \cite{Amann05} and is presented in Section \ref{SectionStationary3}. For earlier investigations of the direction of deflection, we refer to \cite{EL17, LW14a,Nik21b}. \\

Concerning the small aspect ratio model \eqref{A.56}, the previous results, Theorem \ref{exstatsL}-Theorem \ref{DeflectionOut}, remain true. Additionally, we present a rigorous investigation of the direction of deflection for the inner catenoid $u_{in}$. A simplified version reads:
\begin{thm}{\bf (Direction of Deflection in \eqref{A.56})}\label{Auslenkung16simp} \\
Let $\sigma>\sigma_{crit}$ be fixed. Then, there are $\sigma_*,\sigma^*$ with\\
{\bf(i)} If $\sigma < \sigma_*$, then there exists $\delta > 0$ such that
\begin{align*}
u_{in}^{\overline{\lambda}} (z) > u_{in}^{\lambda}(z) \,, \qquad 0 \leq \overline{\lambda} < \lambda < \delta \,, \quad z \in (-1,1)\,.
\end{align*} 
{\bf(ii)} If $\sigma > \sigma^*$, then there exist $\delta > 0$ such that $u_{in}^{\overline{\lambda}}$ and $u_{in}^{\lambda}$ intersect exactly two times for $0 \leq \overline{\lambda} < \lambda < \delta$.
\end{thm}
Part (i) means that the unstable stationary solutions $u^\lambda_{in}$ deflects inwards instead of outwards which confirms formal results from \cite{MP09}.  The precise statement and its proof is content of Section \ref{SectionStationary4}. It is based on an anti-maximums-principle \cite{Shi04}, see Appendix \ref{AMP}.

\section{Notations and Preliminaries}\label{NP}

Let $q \in (1,\infty)$ and $s\in (0,2]$ with $s \neq 1/q$. Put
\begin{align*}
 W^s_{q,D}(-1,1):= \begin{cases} \ \ \ W^s_q(-1,1) \quad \qquad \qquad \qquad \qquad \ \ \,\text{for} \quad s\in (0,1/q)\,, \\
  \ \ \big \lbrace f \in W^s_q(-1,1) \, \big \vert \, f(\pm1)=0 \, \big \rbrace \quad\,\text{for} \quad s\in (1/q,2]\,,
\end{cases} 
\end{align*}
where $W^s_q(-1,1)$ is the fractional Sobolev space over $L_q(-1,1)$ of order $s$. We write $A \in \mathcal{H}  \big(W^2_{q,D}(-1,1), L_q(-1,1) \big )$
if $-A$ generates an analytic semigroup on $L_q(-1,1)$ with domain $W^2_{q,D}(-1,1)$, see \cite{AmannLQPP}. If $E_1$ and $E_2$ are Banach spaces, we denote by $\mathcal{L}(E_1,E_2)$ the Banach space of bounded linear operators from $E_1$ to $E_2$. Moreover, we write $E_1 \hookrightarrow E_2$ if $E_1$ is continuously embedded in $E_2$.\\

In the following, it is convenient to introduce
  \begin{align*}
S:=\big \lbrace w\in W^2_{q,D}(-1,1) \ \big \vert \, -1 < w < 1\big \rbrace 
 \end{align*}
 and
 \begin{align}
 F(w):= \sigma \partial_z \mathrm{arctan}(\sigma \partial_z w) - \frac{1}{w+1} \,, \qquad w \in S\,, \label{defF}
 \end{align}
 so that the stationary free boundary problem \eqref{eq11} becomes
 \begin{align}
  F(w)+\lambda g(w) =0 \,, \qquad w \in S\,.\label{statlambda1}
 \end{align}

\begin{section}{Existence: Proof of Theorem \ref{exstatsL}.}\label{SectionStationary1}

{\bf Proof of Theorem \ref{exstatsL}.}
We resolve \eqref{statlambda1} locally around $(w, \lambda)=(u_{in},0)$ and $(w,\lambda)=(u_{out},0)$. Because $F$ and $g$ (see \cite[Proposition 3.1]{LSS24b}) are analytic from $S$ to $L_q(-1,1)$,
this is possible if and only if $DF(u_{in})$ and $DF(u_{out})$ are isomorphisms from $W_{q,D}^2(-1,1)$ to $L_q(-1,1)$. 
A direct computation shows
\begin{align}
 DF(u_*) v = \frac{\sigma^2}{\mathrm{cosh}^2(c z)} \partial_{zz} v - \frac{2\sigma^2\,c}{\mathrm{cosh}^2(cz)} \mathrm{tanh}(c z) \partial_z v + \frac{\sigma^2\,c^2}{\mathrm{cosh}^2(cz)}v\,, \qquad v \in W^2_{q,D}(-1,1) \label{defB}
\end{align}
with $u_*$ defined in \eqref{catenoid}. Now, $DF(u*)$ is an isomorphism if and only if $DF(u_*)v=0$ has the unique solution $v=0$ in $W^2_{q,D}(-1,1)$. Multiplying $DF(u_*)v=0$ by $-\frac{\mathrm{cosh}^2(cz)}{\sigma^2}$ yields the equivalent condition that
 \begin{align}
 \begin{cases}
 -\partial_{zz} v +2 \, c \, \mathrm{tanh}(cz)\partial_z v-c^2 v = 0\,,\\
 v(\pm1)=0 
 \end{cases} \label{shooting1}
 \end{align}
 only possesses the trivial solution for $c$ equal to $c_{in}$ or $c_{out}$.
 This has already been shown in \cite[p.\,49]{Moulton08} with the aid of the shooting method, which is briefly recalled here: First, one fixes $C_1, C_2 \in \mathbb{R}$ and observes that the initial value problem
   \begin{align*}
 \begin{cases}
 -\partial_{zz} v +2 \, c \, \mathrm{tanh}(cz)\partial_z v-c^2 v = 0\,,\\
 v(0)=C_1 \,, \qquad \partial_z v(0)=C_2 \, c
 \end{cases} 
 \end{align*}  
 has the unique solution
 \begin{align}
 v(z)= C_2 \, \mathrm{sinh}(cz)- C_1 \, \big (c\,z\,\mathrm{sinh}(cz)-\mathrm{cosh}(cz) \big )\,.\label{shooting2}
\end{align}
Next, one adjust $C_1$ and $C_2$ such that $v$ satisfies the boundary conditions in \eqref{shooting1}, and thereby one finds that
\begin{subequations}
\begin{equation}
\text{\eqref{shooting1} has only the trivial solution for $c \neq c_{crit}$,} \label{firsteig0}
\end{equation}
 while
\begin{align}
 &v(z)= C_1 \,  \big (c_{crit}\,z\,\mathrm{sinh}(c_{crit}z)-\mathrm{cosh}(c_{crit}z) \big )\,, \quad C_1 \in \mathbb{R} \setminus \lbrace 0 \rbrace \notag \\
 &\text{is a non-trivial solution to \eqref{shooting1} for $c =c_{crit}$.}\label{firsteig}
\end{align}
\end{subequations}
 Since $c_{in} > c_{crit} > c_{out}$, we find that $DF(u_{in})$ as well as $DF(u_{out})$ are isomorphisms between $W^2_{q,D}(-1,1)$ and $L_q(-1,1)$. 
Hence, the implicit function theorem in the form \cite[Theorem 4.5.4]{BT03} yields some $\delta >0$ and analytic functions 
\begin{align*}
[\lambda \mapsto u^{\lambda}_{in}] &: [0,\delta) \rightarrow W^2_{q,D}(-1,1) \,, \ \qquad u_{in}^0=u_{in}\,,\\
[\lambda \mapsto u^{\lambda}_{out}] &: [0,\delta) \rightarrow W^2_{q,D}(-1,1) \,, \qquad u_{out}^0=u_{out}
\end{align*}
such that $u^\lambda_{in}$ and $u^\lambda_{out}$ are two different solutions to \eqref{eq11} for each $\lambda \in (0,\delta)$. 
The symmetry of $u^\lambda_{out}$ and $u^\lambda_{in}$ is shown similarly as in \cite[Theorem 1.1]{LSS24b}.
\qed

\begin{bem}\label{notcat}
In agreement with Theorem \ref{exstatsL} we denote from now on the inner catenoid $u_{in}$ by $u_{in}^0$, the outer catenoid $u_{out}$ by $u_{out}^0$, and the generic catenoid $u_*$ from \eqref{catenoid} by $u_*^0$.
\end{bem}

\end{section}

\begin{section}{Stability: Proof of Theorem \ref{stability}}\label{SectionStationary2}

\subsection{Stability Analysis of the Inner and Outer Catenoid}
First, we study stability of $u^0_{in}$ and $u^0_{out}$, i.e. the special case $\lambda =0$ in Theorem \ref{stability}. To this end, fix $\sigma > \sigma_{crit}$ and set $\lambda =0$. For a uniform computation, we linearize the dynamical version of \eqref{eq11}, see \eqref{dyn}, around 
\begin{align*}
 u_{*}^0(z)= \frac{\mathrm{cosh}(cz)}{\mathrm{cosh}(c)}-1 
\end{align*}
 with $c$ being either $c_{in}$ or $c_{out}$. For a solution $u \in W^2_{q,D}(-1,1)$ to the dynamical version of \eqref{eq11}, see \eqref{dyn}, with initial value $u_0$ close to $u_{*}^0$, we put $v:=u-u_{*}^0$ and write
\begin{align*}
 \partial_t v = \partial_t(u-u_{*}^0) 
  =F(u_{*}^0+v) - F(u_{*}^0)\,
 \end{align*}
with $F$ given by \eqref{statlambda1} and being smooth in a $W^2_{q,D}$-neighbourhood of $u_{*}^0$. We recall from \eqref{defB} that
\begin{align}
 DF(u_{*}^0) v &= \frac{\sigma^2}{\mathrm{cosh}^2(c z)} \partial_{zz} v - \frac{2\sigma^2\,c}{\mathrm{cosh}^2(cz)} \mathrm{tanh}(c z) \partial_z v + \frac{\sigma^2\,c^2}{\mathrm{cosh}^2(cz)}v \nonumber\\
 &= \sigma^2\, \bigg [ \partial_z \Big ( \frac{1}{\mathrm{cosh}^2(cz)} \,\partial_z v \Big ) + \frac{c^2}{\mathrm{cosh}^2(cz)} v \bigg ]\,. \label{defB2}
\end{align}
Thus, the linearization of \eqref{eq11} around the generic catenoid $u_{*}^0$ is given by
\begin{align*}
 \partial_t v - DF(u_{*}^0) v = F(v+u_{*}^0) - F(u_{*}^0) -DF(u_{*}^0) v =:G(v)
\end{align*}
with $DF(u_{*}^0)$ as above and $G \in  C^\infty\big ( \mathcal{O}, L_q(-1,1) \big )$ for a small neighbourhood $\mathcal{O}$ of $0$ in $W^2_{q,D}(-1,1)$ satisfying $G(0)= 0$ as well as $DG (0)=0$. 
Moreover, since $-DF(u_{*}^0)$ is a uniformly elliptic operator of second order with bounded smooth coefficients, $-DF(u_{*}^0)$ belongs to $\mathcal{H}\big (W^2_{q,D}(-1,1),L_q(-1,1)\big)$, see \cite[Theorem 2.5.1\,(ii)]{LLMP04}. 
Letting $\mu_0(c)$ be the first eigenvalue of $DF(u_{*}^0)$ $-$ for details on the spectrum of $DF(u_{*}^0)$, we refer to Lemma \ref{SturmLiouville} below $-$ the stability criterion \cite[Theorem 9.1.2, Theorem 9.1.3]{Lunardi95} takes the form: 
\begin{itemize}
\item if $\mu_0(c) < 0$, then $u_{*}^0$ is exponentially asymptotically stable,
\item if $\mu_0(c) > 0$, then $u_{*}^0$ is unstable.
\end{itemize}
As only the sign of this first eigenvalue is crucial, we can substitute $0=\big (\mu - DF(u_{*}^0)\big )v$ by
\begin{align}
 \begin{cases}
  &0=\displaystyle\mu v - \frac{c^2}{\mathrm{cosh}^2(cz)} v - \partial_z \Big ( \frac{1}{\mathrm{cosh}^2(cz)} \partial_z v  \Big )\,,\\
  &v(\pm 1)=0\,.
  \end{cases} \label{stabcat}
  \end{align}
  
Since, this is a regular Sturm-Liouville problem, the following is known:

\begin{lem}\label{SturmLiouville}
For fixed $c \in (0,\infty)$, the spectrum of \eqref{stabcat} consists only of countably infinitely many, algebraically simple eigenvalues 
 $$\mu_0(c) > \mu_1(c) >  \dots > \mu_n(c)\rightarrow -\infty\,.$$
 The normalized eigenfunction $v_n^c$ corresponding to $\mu_n(c)$ has exactly $n$ zeroes in $(-1,1)$ and satisfies 
 $$v_n^c(-z)=(-1)^n\,v_n^c(z)\,, \qquad z \in(-1,1)\,.$$
\end{lem}
\begin{proof} This follows from \cite[p.\,286]{Walter00}, except for the fact that each eigenvalue is semi-simple in the sense \cite[Definition A.2.3]{Lunardi95} that follows from a direct computation.
\end{proof}

The function $[c \mapsto \mu_0(c)]$ is called {\it first eigencurve} for \eqref{stabcat}. In \cite{BV96}, qualitative properties of eigencurves for Sturm-Liouville problems depending linearly on a parameter $c$ are stated. 
Though \eqref{stabcat} depends non-linearly on $c$, it is still possible to adapt \cite[Section 2.1]{BV96}:

\begin{prop}\label{AnEigKur}
 The first eigencurve
 $$\mu_0: (0,\infty) \rightarrow \mathbb{R} \,, \qquad c \mapsto \mu_0(c)$$
of \eqref{stabcat} is smooth and has exactly one zero. It is attained at $c_{crit}$ with $\mu_0^\prime(c_{crit})>0$.
\end{prop}

\begin{proof}
{\bf(i)} {\it Smoothness:} 
Let $v(\, \cdot \,; c,\mu)$ be the unique non-trivial solution to 
 \begin{align}
 0=\mu v - \frac{c^2}{\mathrm{cosh}^2(cz)} v - \partial_z \Big ( \frac{1}{\mathrm{cosh}^2(cz)} \partial_z v \Big ) \label{stabcatAWP}
  \end{align}
supplemented with initial conditions 
\begin{align}
v(-1)=0\,, \qquad \partial_z v(-1)=1\,, \label{intcon}
\end{align} 
and define
\begin{align}
 D(c,\mu):=v(1;c,\mu)\,. \label{defD}
\end{align}
As $v(\, \cdot \,; c,\mu)$ depends smoothly on the parameters $(c,\mu)$, see for example \cite[Theorem~9.5, Remark~9.6\,(b)]{Amann95}, we have $D \in C^\infty\big( (0,\infty)\times \mathbb{R},\mathbb{R}\big )$. Moreover, we note that $\mu$ and  $v(\, \cdot \,; c,\mu)$ are a pair of eigenvalue and eigenfunction to \eqref{stabcat} if and only if $D(c, \mu) =0$. We claim that it is further possible to characterize the first eigenvalue $\mu_0(c)$ via $D$ and  $v(\, \cdot \,; c,\mu)$:
\begin{align}
 D(c,\mu) =0 \ \text{and} \ v( z; c,\mu) \neq 0 \ \text{for} \ z \in (-1,1) \qquad \Longleftrightarrow \qquad \mu=\mu_0(c)\,. \label{D0}
\end{align}
Indeed, if $D(c,\mu)=0$ and $v( z; c,\mu) \neq 0 \ \text{for} \ z \in (-1,1)$, then $v(\, \cdot\,; c,\mu)$ is an eigenfunction of \eqref{stabcat} corresponding to the eigenvalue $\mu$ and having no zero in $(-1,1)$. It then follows from Lemma \ref{SturmLiouville} that $\mu=\mu_0(c)$. 
Otherwise, if $\mu$ coincides with the first eigenvalue $\mu_0(c)$ of \eqref{stabcat}, then the unique solvability of initial value problems yields a constant $C \in \mathbb{R} \setminus \lbrace 0 \rbrace$ with
\begin{align*}
 v\big (\, \cdot \,, c , \mu_0(c) \big ) = C  v_0^c\,,
\end{align*}
where $v_0^c$ denotes the first eigenfunction from Lemma \ref{SturmLiouville}. Thus, by Lemma \ref{SturmLiouville} the function $ v\big (\, \cdot \,, c , \mu_0(c) \big )$ satisfies Dirichlet boundary conditions and has no zero in $(-1,1)$. This proves \eqref{D0}.\\

For fixed $c >0$, we wish to resolve $D (c,\mu ) =0$ for $\mu$ locally around $(c,\mu)=\big (c, \mu_0(c) \big )$.
Recalling that $v=v(\,\cdot\,; c,\mu)$  depends smoothly on $\mu$ and $c$, we compute that the derivative of \eqref{stabcatAWP} with respect to $\mu$ is given by
\begin{align}
0=v+\mu \partial_\mu  v- \frac{c^2}{\mathrm{cosh}^2(cz)} \partial_\mu v - \partial_z \left ( \frac{1}{\mathrm{cosh}^2(cz)} \partial_{z}\partial_{\mu} v\right )\,.\label{231023}
\end{align} 
Multiplying \eqref{stabcatAWP} by $\partial_\mu v$ and subtracting the product of \eqref{231023} and $v$, we find
\begin{align*}
0
&=-v^2 - \partial_z \Big ( \frac{1}{\mathrm{cosh}^2(cz)} \partial_z v \Big ) \partial_\mu v + \partial_z \Big ( \frac{1}{\mathrm{cosh}^2(cz)} \partial_{z}\partial_{\mu} v\Big ) v \,.
\end{align*}
Integrating the previous identity over $(-1,1)$ yields
\begin{align}
0 < \int_{-1}^1 v^2 \, \mathrm{d} z &= \int_{-1}^1 \bigg (\partial_z \Big ( \frac{1}{\mathrm{cosh}^2(cz)} \partial_{z}\partial_{\mu}v \Big ) v - \partial_z \Big ( \frac{1}{\mathrm{cosh}^2(cz)} \partial_z v \Big ) \partial_\mu v\bigg) \mathrm{d} z  \nonumber \\
&= \bigg [ \frac{1}{\mathrm{cosh}^2(cz)}  ( \partial_{z}\partial_{\mu} v - (\partial_z v)(\partial_\mu v) )\bigg ]^{z=1}_{z=-1} \,. \label{231023b}
\end{align}
We want to evaluate \eqref{231023b} at $(c, \mu)= \big (c,\mu_0(c) \big )$. For $\mu=\mu_0(c)$, it follows from \eqref{D0} that $v\big ( \, \cdot \,;c,\mu_0(c) \big )$ is a first eigenfunction, and Lemma \ref{SturmLiouville} yields that $v( \, \cdot \,;c,\mu_0(c) \big )$ is even with $v\big (\pm 1; c, \mu_0(c) \big )=0$. By symmetry and the initial conditions \eqref{intcon}, we get $\partial_z v (1;c,\mu_0(c)  )= -\partial_z v (-1; c, \mu_0(c) )=-1$. Moreover, applying the initial condition $v(-1;c,\mu)=0$ for all $(c,\mu)$, we find $\partial_\mu v \big (-1;c , \mu_0(c) \big )=0$. Consequently, \eqref{231023b} can be reduced to 
\begin{align*}
0 < \int_{-1}^1 v^2 \, \mathrm{d} z = \frac{\partial_\mu v \big (1; c,\mu_0(c) \big )}{\mathrm{cosh}^2(c)} \,.
\end{align*}
Recalling that $D (c,\mu )=v\big (1;c,\mu )$ by \eqref{defD}, we deduce further that
\begin{align}
\partial_\mu D \big (c,\mu_0(c) \big ) = \partial_\mu v \big (1;c,\mu_0(c) \big ) = \mathrm{cosh}^2(c) \int_{-1}^1 v^2 > 0\,. \label{ablmu}  
\end{align}
Hence, for fixed $c >0$, the implicit function theorem yields some $\rho > 0$ and a function $\tilde{\mu} \in C^\infty \big ((c-\rho, c+\rho), \mathbb{R} \big)$ with $\tilde{\mu}(c)=\mu_0(c)$ and 
\begin{align}
 D\big (\tilde{c},\tilde{\mu}(\tilde{c}) \big )=D \big (c,\mu_0(c) \big )= 0 \,, \qquad \tilde{c} \in (c-\rho,c+\rho)\,. 
\end{align}

In addition, by the smooth dependence of $v\big (\, \cdot \,, \tilde{c}, \tilde{\mu}(\tilde{c}) \big )$ on $\tilde{c}$, we may assume that $v\big (\, \cdot \,, \tilde{c}, \tilde{\mu}(\tilde{c}) \big )$ 
has no zero in $(-1,1)$ as the same holds true for $v\big (\, \cdot \,, c, \mu_0(c) \big )$. Thus, \eqref{D0} implies 
\begin{align*}
 \mu_0(\tilde{c})=\tilde{\mu}(\tilde{c}) \,, \qquad \tilde{c} \in (c-\rho,c+\rho)\,,
\end{align*}
and the smoothness of $[c \mapsto \mu_0(c)]$ follows from that.\\

{\bf(ii)} {\it Zeroes:} Rewriting \eqref{stabcat} for $\mu =0$ in non-divergence form, we see that it is equivalent to \eqref{shooting1}.
Hence, it follows from \eqref{firsteig0} and \eqref{firsteig} that $0$ is an eigenvalue of \eqref{stabcat} if and only if $c=c_{crit}$. In this case, the corresponding eigenfunction is a multiple of 
\begin{align*}
w(z):=c_{crit} \,z\, \mathrm{sinh}(c_{crit} z) - \mathrm{cosh}(c_{crit}  z) \,.
\end{align*}
Since $w$ has no zeroes in $(-1,1)$, we deduce from Lemma \ref{SturmLiouville} that $0$ is the first eigenvalue of \eqref{stabcat} for $c=c_{crit}$ so that $c_{crit}$ is indeed the only zero of $\mu_0$. \\

{\bf(iii)} {\it Derivative at $c_{crit}$:} Since
\begin{align}
\mu_0^\prime(c_{crit}) = - \frac{\partial_c D(c_{crit},0)}{\partial_\mu D(c_{crit} ,0)} \label{impabl}
\end{align}
and $\partial_{\mu} D(c_{crit},0) > 0$ thanks to \eqref{ablmu}, we have to check that $\partial_c D(c_{crit},0)<0$. Differentiating \eqref{stabcatAWP} with respect to $c$ yields
\begin{align}
0
&=\mu \partial_c v +\frac{2 c^2 \mathrm{sinh}(cz)z}{\mathrm{cosh}^3(cz)} v- \frac{2 c}{\mathrm{cosh}^2(cz)} v  - \frac{c^2}{\mathrm{cosh}^2(cz)} \partial_c v \nonumber \\
&\hphantom{=}\,+ \partial_z \Big ( \frac{2\, \mathrm{sinh}(cz)z}{\mathrm{cosh}^3(cz)} \partial_z v \Big ) - \partial_z \Big ( \frac{1}{\mathrm{cosh}^2(cz)} \partial_{z}\partial_{c}v \Big ) \,. \label{ablc}
\end{align}
Multiplying \eqref{ablc} by $v=v(\,\cdot\,;c,\mu)$ and subtracting the product of \eqref{stabcatAWP} and $\partial_c v$ yields
\begin{align}
0
&= \partial_z \Big ( \frac{1}{\mathrm{cosh}^2(cz)} \partial_z v \Big ) \partial_c v\nonumber \\
&\hphantom{=}\,+ \frac{2c^2 \mathrm{sinh}(cz) z}{\mathrm{cosh}^3(cz)} v^2 - \frac{2c}{\mathrm{cosh}^2(cz)} v^2
+ \partial_z \Big ( \frac{2\,\mathrm{sinh}(cz)z}{\mathrm{cosh}^3(cz)} \partial_z v \Big) v - \partial_z \Big ( \frac{1}{\mathrm{cosh}^2(cz)} \partial_{z} \partial_{c} v \Big ) v \,. \label{291023}
\end{align}
Plugging $(c,\mu)=(c_{crit}, 0)$ into \eqref{291023} and then integrating from $-1$ to $1$ gives 
\begin{align}
 &\int_{-1}^1 \bigg (\partial_z \Big ( \frac{1}{\mathrm{cosh}^2(c_{crit}z)} \partial_z v\Big ) \partial_c v-\partial_z \Big ( \frac{1}{\mathrm{cosh}^2(c_{crit}z)} \partial_{z}\partial_{c} v\Big ) v \bigg )\mathrm{d} z \notag\\
 = &\int_{-1}^1  \frac{2c_{crit}}{\mathrm{cosh}^2(c_{crit}z)} \big (1-c_{crit} \,\mathrm{tanh}(c_{crit}z)z \big ) v^2  \mathrm{d} z + \int_{-1}^1 \frac{2\mathrm{sinh}(c_{crit}z)z}{\mathrm{cosh}^3(c_{crit}z)} (\partial_z v)^2\,\mathrm{d} z\,. \label{intdc}
\end{align}
For the second integral on the right-hand side, we have used integration by parts and the fact that the boundary terms vanish due to $v(\pm 1 ;c_{crit},0)=0$ by \eqref{D0} and $\mu_0(c_{crit})=0$.
From 
\begin{align*}
 1-c_{crit}\mathrm{tanh}(c_{crit}z)z &\geq 1 - c_{crit} \mathrm{tanh}(c_{crit}) \\
 &= \frac{\mathrm{cosh}(c_{crit})-c_{crit}\mathrm{sinh}(c_{crit})}{\mathrm{cosh}(c_{crit})}=0 \,, \qquad z\in (-1,1)\,,
\end{align*}
which is due to \eqref{Defccrit} combined with the positivity of the second integral on the right-hand side of \eqref{intdc}, we deduce that
\begin{align*}
 0&<\int_{-1}^1 \bigg ( \partial_z \Big ( \frac{1}{\mathrm{cosh}^2(c_{crit}z)} \partial_z v\Big ) \partial_c v-\partial_z \Big ( \frac{1}{\mathrm{cosh}^2(c_{crit}z)} \partial_{z}\partial_{c} v\Big ) v \bigg )\mathrm{d} z\\
 &\ = \Big [ \frac{1}{\mathrm{cosh}^2(c_{crit}z)} \big ((\partial_z v)(\partial_cv)-(\partial_z\partial_cv)v \big ) \Big ]_{z=-1}^{z=1}\\
 &\ =\frac{-\partial_c D(c_{crit},0)}{\mathrm{cosh}^2(c_{crit})}\,,
\end{align*}
where we have used $\partial_z v(1;c_{crit},0)= -\partial_z v(-1;c_{crit},0)$ by symmetry, the initial values \eqref{intcon} and the definition of $D$.
Finally, \eqref{ablmu} and \eqref{impabl} yield $\mu_0^\prime(c_{crit}) > 0$.
\end{proof}

\begin{cor}\label{Stabcatcor}
The inequalities $\mu_0(c_{out})< 0$ and $\mu_0(c_{in}) > 0 >\mu_1(c_{in})$ hold true. 
\end{cor}

\begin{proof}
This follows from Proposition \ref{AnEigKur} and the fact that $c_{out} < c_{crit} < c_{in}$, see \eqref{crel}. Note that similar arguments as in step (i) of the proof of Proposition \ref{AnEigKur} guarantee the smoothness of the second eigencurve $[c \mapsto \mu_1(c)]$, which always lies below the first eigencurve $[c \mapsto \mu_0(c)]$. Because the first eigencurve is sign-changing and
the only eigencurve with a zero by step (ii) in the proof  of Proposition \ref{AnEigKur}, it follows that $0 > \mu_1(c_{in})$.
\end{proof}

In particular, $DF(u_{out}^0)$ has only strictly negative eigenvalues, while $DF(u_{in}^0)$ has exactly one strictly positive eigenvalue and all other eigenvalues are strictly negative.
Regarding the stability analysis of the catenoids, we end up with the following:

\begin{cor}
For $\sigma > \sigma_{crit}$ and $\lambda=0$, the inner catenoid $u_{in}^0$ is unstable whereas the outer catenoid $u_{out}^0$ is exponentially asymptotically stable in $W^2_{q,D}(-1,1)$. 
\end{cor}

Finally, we come to our main purpose and show the corresponding properties of $u^\lambda_{in}$ and $u^\lambda_{out}$ for $\lambda>0$ sufficiently small:

\subsection{Proof of Theorem \ref{stability}.}
Letting $u^\lambda_*$ be either $u^\lambda_{in}$ or $u^\lambda_{out}$, the linearization of the dynamical version of \eqref{eq11}, see \eqref{dyn}, around $u^\lambda_*$ reads
\begin{align}
\partial_t v -\big (DF(u^\lambda_*)+ \lambda Dg(u^\lambda_*)\big) v &=  F(u^\lambda_* + v)- F(u^\lambda_*) - DF(u^\lambda_*)v \nonumber \\
&+\lambda \big (g(u^\lambda_* + v)- g(u^\lambda_*) - Dg(u^\lambda_*)v \big )=:G_\lambda(v)\,, \label{lin3}
\end{align}
where $F$ is given by \eqref{defF}. Thanks to \cite[Proposition 3.1]{LSS24b}, we find $G_\lambda \in C^\infty\big ( \mathcal{O}, L_q(-1,1) \big )$ 
for a small neighbourhood $\mathcal{O}$ of $0$ in $W^2_{q,D}(-1,1)$ satisfying $G_\lambda(0)= 0$ as well as $DG_\lambda(0)=0$. Moreover, since 
\begin{align*}
\Vert DF(u_*^\lambda) &+\lambda Dg(u_*^\lambda) - DF(u_{*}^0) \Vert_{\mathcal{L}(W^2_{q,D},L_q)} \\
&\leq \Vert DF(u_*^\lambda) - DF(u_*^0) \Vert_{\mathcal{L}(W^2_{q,D},L_q)} +\lambda \Vert  Dg(u_*^\lambda) \Vert_{\mathcal{L}(W^2_{q,D},L_q)} \rightarrow 0 \,,
\end{align*}
as $\lambda \rightarrow 0$ by Theorem \ref{exstatsL}, and $-DF(u_{*}^0) \in \mathcal{H}\big (W^2_{q,D}(-1,1),L_q(-1,1) \big)$, we deduce from \cite[Theorem 1.3.1\,(i)]{AmannLQPP} the existence of $\delta> 0$ such that  
$$-\big (DF(u_*^\lambda) +\lambda Dg(u_*^\lambda)\big )\in \mathcal{H}\big (W^2_{q,D}(-1,1),L_q(-1,1) \big)\,, \qquad \lambda \in  [0,\delta)\,.$$
We now investigate the stability of $u^\lambda_{in}$ and $u^\lambda_{out}$ separately:\\
 
{\bf(i)} {\it Instability of $u^\lambda_{in}$}: Due to Corollary \ref{Stabcatcor} and Lemma \ref{SturmLiouville}, the operator $DF(u_{in}^0)$ possesses a positive, isolated and algebraically simple eigenvalue so that the perturbation result \cite[Proposition A.3.2]{Lunardi95} 
for such eigenvalues allows to make $\delta > 0$ smaller such that $DF(u_{in}^\lambda) +\lambda Dg(u_{in}^\lambda)$ also has an eigenvalue with positive real part for $\lambda \in [0,\delta)$. Moreover, 
since the embedding $W^2_{q,D}(-1,1) \hookrightarrow L_q(-1,1)$ is compact, the spectrum of $DF(u_{in}^\lambda) +\lambda Dg(u_{in}^\lambda)$ consists only of eigenvalues with no finite accumulation point, see \cite[Theorem 6.29]{Kato95}. Thus, there is a constant $C>0$ such that the strip $\big \lbrace \mu \in \mathbb{C} \, \big \vert \, 0 < \mathrm{Re} \mu \, < C \big \rbrace$
is contained in the resolvent set of $DF(u^\lambda_{in})+\lambda Dg(u^\lambda_{in})$. Applying now \cite[Theorem 9.1.3]{Lunardi95} to \eqref{lin3} shows the instability of $u^\lambda_{in}$. \\

{\bf(ii)} {\it Stability of $u^\lambda_{out}$}: Since the spectral bound of $DF(u_{out}^0)$ is negative due to Corollary \ref{Stabcatcor}, it follows from \cite[Corollary 1.4.3]{AmannLQPP} that we may take $\delta >0$ so small that $DF(u_{out}^\lambda) +\lambda Dg(u_{out}^\lambda)$ also has a negative spectral bound for $\lambda \in [0,\delta)$. Hence, \cite[Theorem 9.1.2]{Lunardi95} implies that $u^\lambda_{out}$ is exponentially asymptotically stable.
\qed

\begin{bems}\label{CoV}
For $c_{out} < c_{crit}$, it is possible to apply the comparison principle for eigenvalues of Sturm-Liouville problems \cite[p.\,294]{Walter00} to get that $\mu_0(c_{out}) < \mu_0(c_{crit})=0$. The same result also follows from computing the second variation of the surface energy 
$$E_m(u)= \int_{-1}^1 (u+1) \,\sqrt{1+\sigma^2 (\partial_z u)^2} \,\mathrm{d} z \,, \qquad u(\pm1)=0$$
in $u_{out}^0$. However, both approaches do not apply to $\mu_0(c_{in})$.
\end{bems}

\end{section} 
\begin{section}{Direction of Deflection: Proof of Theorem \ref{DeflectionOut}}\label{SectionStationary3}

\subsection{Ansatz} 
Because $u^\lambda_{out}$ was constructed in Theorem \ref{exstatsL} by applying the implicit function theorem to the analytic function $\big[w \mapsto F(w)+\lambda g(w) \big]$, we may write
\begin{align*}
 u^\lambda_{out}= u_{out}^0 +\lambda \, \partial_\lambda u^0_{out} + o(\lambda) \,, \qquad \lambda \rightarrow 0 
\end{align*}
with
\begin{align}
 \partial_\lambda u^0_{out} = - \big (DF(u_{out}^0)\big )^{-1}\,g(u_{out}^0) \label{Auslenkung5}
\end{align}
in $W^2_{q,D}(-1,1)$. Here, $g$ is the electrostatic force, and $u_{out}^0$ is the outer catenoid. We recall from \eqref{defB2} that
\begin{align} DF(u_{out}^0) v= \sigma^2 \left [ \partial_z \Big( \frac{1}{\mathrm{cosh}^2(c_{out}z)} \partial_z v \Big ) + \frac{c_{out}^2}{\mathrm{cosh}^2(c_{out}z)}\,v\right ]\,, \label{261023a}
\end{align}
as well as $g(u_{out}^0)(z) \geq 0$, $z \in (-1,1)$ by \eqref{eq2.5}. Thus, the sign of 
\begin{align} u^\lambda_{out} - u_{out}^0= \lambda \big(-DF(u_{out}^0) \big)^{-1} g(u_{out}^0) +o(\lambda) \,, \qquad \lambda \rightarrow 0 \label{Auslenkung5a}
\end{align}
for small $\lambda$ is decided by positivity properties of $DF(u_{out}^0)$. Note that the scalar function $-\frac{c^2_{out}}{\mathrm{cosh}^2(c_{out}z)} <0$
appearing in the definition of $-DF(u_{out}^0)$ has the wrong sign for the common weak and strong maximum principles \cite[Theorem 6.4.2, Theorem 6.4.4]{Evans10} to apply. Instead, as $-DF(u_{out}^0)$ is of the form \eqref{261023a}, it falls in the class of operators investigated in \cite{Amann05}, and we can rely on a strong maximum principle from \cite{Amann05}. It is based on functional analysis and requires that $DF(u_{out}^0)$ has a negative spectral bound, which is true thanks to Corollary \ref{Stabcatcor}.

\begin{lem}\label{BqMax}
Let $f \in L_q(-1,1)$ with $f \geq 0$ a.e. and $f \not\equiv 0$. Then, the function 
$$v:=\big ( -DF(u_{out}^0) \big )^{-1} f \in W^2_{q,D}(-1,1)$$
satisfies 
$v(z) > 0$ for $z \in (-1,1)$ as well as $\partial_z v(-1) > 0$ and $\partial_z v(1) < 0$ . 
\end{lem}

\begin{proof}
Recall that $q > 2$, hence $W_{q,D}^2(-1,1) \hookrightarrow C^1([-1,1])$, and that the spectrum of $DF(u_{out}^0)$ is contained in $\big (-\infty, 0)$ thanks to Corollary \ref{Stabcatcor}. Now \cite[Theorem 15]{Amann05} yields the assertion.
\end{proof}

We check that the right-hand side $g(u_{out}^0)$ satisfies the conditions of the above lemma:

\begin{lem}\label{gnontrivial}
The function $g(u_{out}^0)$ belongs to $L_q(-1,1)$ with $g(u_{out}^0)\geq 0$ a.e. and $g(u_{out}^0) \not\equiv 0$.
\end{lem}

\begin{proof}
This follows from \eqref{eq2} and Hopf's Lemma.
\end{proof}

\subsection{Proof of Theorem \ref{DeflectionOut}} The proof is similar to \cite[Theorem 1.3]{LSS24b}. From Lemma \ref{BqMax}, Lemma \ref{gnontrivial} and \eqref{Auslenkung5}, it follows that $\partial_z[\partial_\lambda u_{out}^0 ](1)< 0$ as well as $\partial_z[\partial_\lambda u_{out}^0 ](-1)>0$. Thanks to the embedding of $W^2_{q,D}(-1,1)$ in $C^1([-1,1])$\,, we find $\varepsilon > 0$ such that 
\begin{align}
\partial_z[\partial_\lambda u_{out}^0 ] (z) &\leq -4 \varepsilon \,, \qquad z \in (1-\varepsilon,1]\,, \notag \\
\partial_z[\partial_\lambda u_{out}^0 ] (z) &\geq\, 4 \varepsilon \,,\ \  \qquad z \in (-1,-1+\varepsilon]\,. \label{Auslenkung6}
\end{align}
Furthermore, since $\partial_\lambda u_{out}^0$ is continuous and strictly positive on $[-1+\varepsilon,1-\varepsilon]$ by Lemma \ref{BqMax}, we find $\tilde{\varepsilon} > 0$ such that
\begin{align}
\partial_\lambda u^0_{out}(z) \geq 4 \tilde{\varepsilon} \,, \qquad z \in [-1+\varepsilon,1-\varepsilon]\,. \label{Auslenkung7}
\end{align}
Finally, the continuity of $\big [(z,\lambda) \rightarrow \partial_\lambda u^\lambda_{out}(z)\big ]$ and $\big [(z,\lambda) \rightarrow \partial_z[\partial_\lambda u^\lambda_{out}](z)\big ]$ allows us to extend \eqref{Auslenkung6} and \eqref{Auslenkung7} to 
\begin{align}
\partial_z[\partial_\lambda u_{out}^\lambda ] (z) &\leq - 2\varepsilon \,, \qquad z \in (1-\varepsilon,1] \,, \qquad \  \lambda \in [0,\delta]\,,\notag \\
\partial_z[\partial_\lambda u_{out}^\lambda ] (z) &\geq\,  2\varepsilon \,,\ \  \qquad z \in [-1,-1+\varepsilon)\,, \ \ \, \lambda \in [0,\delta]\,, \label{Auslenkung8}
\end{align}
and 
\begin{align}
\partial_\lambda u^\lambda_{out}(z) \geq  2\tilde{\varepsilon} \,, \qquad z \in [-1+\varepsilon,1-\varepsilon]\,, \qquad \lambda \in [0,\delta]\,, \label{Auslenkung9}
\end{align}
for suitably chosen $\delta>0$. Let us now write 
\begin{align}
u^\lambda_{out} = u^{\overline{\lambda}}_{out} + \partial_\lambda u_{out}^{\overline{\lambda}}\, (\lambda-\overline{\lambda}) + R(\lambda,\overline{\lambda}) \label{tail}
\end{align}
in $W^2_{q,D}(-1,1) \hookrightarrow C^1\big ([-1,1]\big)$ with error term 
\begin{align*}
R(\lambda, \overline{\lambda} ):= \int_0^1 (1-t) \, \partial_\lambda^2\, u^{\overline{\lambda} +t(\lambda-\overline{\lambda})}_{out}  \,\mathrm{d} t \,(\lambda-\overline{\lambda})^2
\end{align*}
satisfying the uniform estimate
\begin{align*}
\frac{\ \ \Vert R(\lambda,\overline{\lambda})\Vert_{C^1}}{\vert \lambda-\overline{\lambda} \vert} \leq C \,\vert \lambda - \overline{\lambda} \vert 
\end{align*}
for some $C > 0$ independent of $\lambda\,, \, \overline{\lambda} \in [0, \delta]$. As a consequence, we can make $\delta >0$ smaller such that
\begin{align}
\frac{\ \ \Vert R(\lambda,\overline{\lambda})\Vert_{C^1}}{\vert \lambda-\overline{\lambda} \vert} \leq \mathrm{min} \big \lbrace \varepsilon, \tilde{\varepsilon} \big \rbrace \,, \qquad 0 < \lambda- \overline{\lambda}  \leq \delta\,, \quad \lambda \leq \delta \,. \label{Auslenkung10}
\end{align}
From \eqref{Auslenkung9}-\eqref{Auslenkung10}, it follows that 
\begin{align*}
\frac{u^\lambda_{out}(z)-u_{out}^{\overline{\lambda}}(z)}{\lambda -\overline{\lambda}}\geq \tilde{\varepsilon}\,, \qquad z \in [-1+\varepsilon, 1-\varepsilon]\,, 
\end{align*}
while \eqref{Auslenkung8} - \eqref{Auslenkung10} yield
\begin{align*}
\frac{\partial_z u^\lambda_{out}(z)-\partial_z u_{out}^{\overline{\lambda}}(z)}{\lambda -\overline{\lambda}}\geq \varepsilon  \,, \qquad z \in [-1,-1+\varepsilon)\,,
\end{align*}
as well as
\begin{align*}
\frac{\partial_z u^\lambda_{out}(z)-\partial_z u_{out}^{\overline{\lambda}}(z)}{\lambda -\overline{\lambda}}\leq -\varepsilon  \,, \qquad z \in (1-\varepsilon,1]\,.
\end{align*}
Here, all three estimates above hold for $0 < \lambda- \overline{\lambda}  \leq \delta$ and $\lambda \leq \delta$. From these estimates and the fact that $$u^\lambda_{out}(\pm 1) = u^{\overline{\lambda}}_{out}(\pm 1)=0\,,$$ we deduce 
\begin{align*}
u^\lambda_{out}(z) > u^{\overline{\lambda}}_{out}(z) \,, \qquad z \in (-1,1) \,, \quad 0 < \lambda- \overline{\lambda}  \leq \delta\,, \quad \lambda \leq \delta. 
\end{align*}
\qed

\section{Additional Results for the Small Aspect Ratio Model \eqref{A.56}}\label{SectionStationary4}

In this subsection, we focus on the small aspect ratio model \eqref{A.56}. The results and proofs of Theorem \ref{exstatsL}-Theorem \ref{DeflectionOut}  remain valid if $g(u)$ is replaced by $g_{sar}(u)$ from \eqref{A.6}.  In particular, there exists again a local curve of unstable stationary solutions $[\lambda \mapsto u^\lambda_{in}]$ emanating from $u_{in}^0$ in the small aspect ratio model. Note that, in general, this curve differs from the curve of stationary solutions emanating from $u_{in}^0$ in the full 
free boundary problem. \\

We aim at understanding in which direction $u^\lambda_{in}$ deflects in the small aspect ratio model \eqref{A.56}. Letting
\begin{align*}
g_{sar}(z):=g_{sar}(u_{in}^0)(z) =  \frac{\mathrm{cosh}^2(c_{in})}{\mathrm{cosh}(c_{in}z)} \frac{1}{\mathrm{ln}^2 \Big ( 2\,\frac{\mathrm{cosh}(c_{in})}{ \mathrm{cosh}(c_{in} z)}\Big )} > 0 \,, \qquad z \in (-1,1)\,,
\end{align*}
our starting point for the investigation of the direction of deflection is again the formula
\begin{align*}
\partial_\lambda u^0_{in} = \big (-DF(u_{in}^0) \big )^{-1} g_{sar}\,,
\end{align*}
which is analogue to \eqref{Auslenkung5}, and we are interested in the sign of $\partial_\lambda u^0_{in}$. Since $c_{in} > c_{crit}$, Corollary~\ref{Stabcatcor} implies now that $DF(u_{in}^0)$ has exactly 
one strictly positive eigenvalue and all other eigenvalues of $DF(u_{in}^0)$ are strictly negative so that the maximum principle from \cite{Amann05} fails. Instead, we apply a criterion for an anti-maximum principle from \cite{Shi04}, see Appendix \ref{AMP}. To this end, let
\begin{align}
\varphi(z):= \mathrm{cosh}(c_{in}z) - c_{in} \, z \, \mathrm{sinh}(c_{in}z) \label{DefPhi}
\end{align}
be the unique solution to the initial value problem
\begin{align}
\begin{cases}
&0 = - \partial_z \Big ( \displaystyle{\frac{1}{\mathrm{cosh}^2(c_{in}z)} \partial_z\varphi \Big ) - \frac{c_{in}^2}{\mathrm{cosh}^2(c_{in}z)}} \varphi \quad \text{on} \ (-1,1)\,,\\
&\varphi(0)= 1 \,, \quad \partial_z\varphi(0)=0\,,
\end{cases} \label{Auslenkung12alt}
\end{align}
associated with the boundary value differential operator $-DF(u_{in}^0)$. 
The function $\varphi$ is symmetric, has exactly two zeroes $z= \pm c_{crit}/c_{in}$ in $(-1,1)$ and is sign-changing. 
With $\varphi$ at hand, the criterion reads: 
\begin{align*}
 &\int_{-1}^1 g_{sar}(z)\varphi(z)\,\mathrm{d} z > 0 \qquad \Longrightarrow \qquad \partial_\lambda u^0_{in} <0\ \text{in}\ (-1,1)\,, \\
 &\int_{-1}^1 g_{sar}(z)\varphi(z)\,\mathrm{d} z < 0 \qquad \Longrightarrow \qquad \partial_\lambda u^0_{in} \ \text{is sign-changing in}\ (-1,1)\,.
\end{align*}

Dependent on the parameter $\sigma$, we get:

\begin{lem}\label{Auslenkung14}
{\bf(i)} There exists $\sigma_*  >\sigma_{crit}$ such that for each $\sigma \in (\sigma_{crit}, \sigma_*)$ the corresponding deflection $[\lambda \mapsto u^\lambda_{in}]$ in the small aspect ratio model \eqref{A.56} satisfies
\begin{align*}
\partial_\lambda u^0_{in}(z) <0 \,, \qquad z \in (-1,1)\,, \qquad \partial_z [\partial_\lambda u^0_{in}](-1) < 0\,, \qquad \partial_z [\partial_\lambda u^0_{in}](1) > 0\,.
\end{align*}
{\bf(ii)} There exists $\sigma^* > \sigma_{crit}$ such that for each $\sigma > \sigma^*$ and each corresponding deflection $[\lambda \mapsto u^\lambda_{in}]$ there exists
$r_0 \in (0,1)$, depending on $\sigma$, such that $\partial_\lambda u^0_{in} < 0$ on $(-r_0,r_0)$ and $\partial_\lambda u^0_{in} > 0$ on $(-1,-r_0) \cup (r_0,1)$ as well as
\begin{align*}
 &\partial_z [\partial_\lambda u^0_{in}](-1)  > 0 \,, \quad \partial_z [\partial_\lambda u^0_{in}](-r_0) < 0 \,, \\
 &\partial_z [\partial_\lambda u^0_{in}](r_0)>0 \,, \quad \partial_z [\partial_\lambda u^0_{in}](1) <0\,.
\end{align*}
Moreover, one has $\sigma^* \geq \sigma_*> \sigma_{crit}$. \color{black}
\end{lem}

\begin{proof}
For simplicity, we use the abbreviation $c=c_{in}$. We write 
 \begin{align}
  \int_{-1}^1 &g_{sar}(z) \,\varphi(z) \, \mathrm{d} z \notag \\
  &= \int_{-1}^1 \bigg (\frac{\mathrm{cosh}^2(c)}{\mathrm{cosh}(cz)} \, \frac{1}{\mathrm{ln}^2 \Big (2\,\frac{\mathrm{cosh}(c)}{\mathrm{cosh}(cz)} \Big )} \big [\mathrm{cosh}(cz) 
  -cz\,\mathrm{sinh}(cz) \big ]\bigg)\mathrm{d} z  \notag \\
  &=  \, \frac{\mathrm{cosh}^2(c)}{c} \, \int_{-c}^c \bigg (\frac{1}{\mathrm{cosh}(z)} \,\frac{1}{\mathrm{ln}^2 \Big ( 2\,\frac{ \mathrm{cosh}(c)}{ \mathrm{cosh}(z)} \Big )} \big [\mathrm{cosh}(z) 
  -z\,\mathrm{sinh}(z) \big ] \bigg) \mathrm{d} z   \notag \\
  &=  2\,\frac{\mathrm{cosh}^2(c)}{c} \,\int_{0}^c \bigg (\frac{1}{\mathrm{ln}^2 \Big ( 2\,\frac{ \mathrm{cosh}(c)}{ \mathrm{cosh}(z)} \Big )} \big [1- z\,\mathrm{tanh}(z) \big ] \bigg ) \mathrm{d} z\notag \\
  &=: 2\,\frac{\mathrm{cosh}^2(c)}{c} \, I_1(\sigma)\,, \label{Auslenkung15}
 \end{align}
where we recall from \eqref{261023b} and \eqref{crel} that $c=c_{in}$ is completely determined by being the largest solution to $\sigma=\frac{\mathrm{cosh}(c)}{c}$. Moreover, note that $2\,\frac{\mathrm{cosh}^2(c)}{c} >0$ is irrelevant for the sign of \eqref{Auslenkung15} and that 
\begin{align}
   1- z\,\mathrm{tanh}(z)\ \,\begin{cases} &\geq 0 \,, \quad z \in (0,c_{crit}]\,, \\        &< 0 \,, \quad z \in (c_{crit},c)\,, \end{cases} \label{Defccrita}
\end{align}
due to the choice of $c_{crit}$ in \eqref{Defccrit}\,. We first estimate $I_1(\sigma)$ from below and then from above:\\

{\bf(i)} From \eqref{Defccrita} we deduce that $I_1(\sigma_{crit}) > 0$. Since the integral $I_1(\sigma)$ depends continuously on $c=c_{in}$, hence continuously on $\sigma \geq \sigma_{crit}$, we find $\sigma_* >\sigma_{crit}$ with
$$I_1(\sigma) > 0 \,, \qquad \sigma \in (\sigma_{crit}, \sigma_*)\,.$$
Thus, \eqref{Auslenkung15} is positive for such $\sigma$ and the assertion follows from Lemma \ref{Auslenkung13}.\\

{\bf(ii)} For the estimate from above, we write
\begin{align*}
  I_1(\sigma)&=  \int_0^{c_{crit}}\bigg ( \frac{1}{\mathrm{ln}^2 \Big ( 2\,\frac{  \mathrm{cosh}(c)}{ \mathrm{cosh}(z)} \Big )} \big [1- z\,\mathrm{tanh}(z) \big ] \bigg) \mathrm{d} z  \\
 &\ \ +  \int_{c_{crit}}^c \bigg (\frac{1}{\mathrm{ln}^2 \Big ( 2\,\frac{  \mathrm{cosh}(c)}{  \mathrm{cosh}(z)} \Big )} \big [1- z\,\mathrm{tanh}(z) \big ] \bigg ) \mathrm{d} z   \\
 &= : I_2(\sigma) + I_3(\sigma)
\end{align*}
and deduce from \eqref{Defccrita} that the integrand in $I_2(\sigma)$ is positive, while the integrand in $I_3(\sigma)$ is negative. Since $\mathrm{ln}\Big (2\,\frac{\mathrm{cosh}(c)}{ \mathrm{cosh}(z)} \Big ) >0$ for all $z \in (0,c)$, we estimate
\begin{align*}
 I_2(\sigma)+I_3(\sigma) &\leq \frac{1}{\mathrm{ln}^2 \Big (2\, \frac{\mathrm{cosh}(c)}{ \mathrm{cosh}(c_{crit})} \Big )} \int_0^{c_{crit}} \big [1-z \, \mathrm{tanh}(z)\big] \, \mathrm{d} z \\
 &\ \ \, + \frac{1}{\mathrm{ln}^2 
 \Big ( 2\,\frac{ \mathrm{cosh}(c)}{\mathrm{cosh}(c_{crit})} \Big )} \int_{c_{crit}}^c \big [ 1-z \, \mathrm{tanh}(z)\big ] \, \mathrm{d} z \\
 &= \frac{1}{\mathrm{ln}^2 \Big ( 2\,\frac{\mathrm{cosh}(c)}{\mathrm{cosh}(c_{crit})} \Big )} \int_0^{c} \big [1-z \, \mathrm{tanh}(z)\big ] \, \mathrm{d} z\,.
\end{align*}
Now the right-hand side is negative if and only if 
\begin{align*}
 I_4(\sigma):= \int_0^c \big [1-z \,\mathrm{tanh}(z) \big ] \, \mathrm{d} z < 0\,.
\end{align*}
Because $\sigma \nearrow \infty$ implies $c=c_{in} \nearrow \infty$, the integral $I_4(\sigma)$ diverges to $-\infty$ and we find $\sigma^* \geq \sigma_{crit}$ such that $I_4(\sigma) <0$ for all $\sigma > \sigma^*$. Hence, \eqref{Auslenkung15} is negative for such values of $\sigma$
and the assertion follows from Lemma \ref{Auslenkung13}.
\end{proof}

Now, we come to the precise version of Theorem \ref{Auslenkung16simp}.
Based on Lemma \ref{Auslenkung14}, we describe the qualitative behaviour of $[\lambda \mapsto u^\lambda_{in}]$ in the small aspect ratio model \eqref{A.56} in case the parameter $\sigma$ is either sufficiently close to $\sigma_{crit}$ or sufficiently large.
The results are depicted in Figure \ref{figurdeflection}. In particular, for $\sigma$ sufficiently close to $\sigma_{crit}$, we discover a contrary behaviour to $u^\lambda_{out}$: The deflection $u^\lambda_{in}$ of $u_{in}^0$ is directed inwards instead of outwards.

\begin{thm}\label{Auslenkung16}
Let $\sigma>\sigma_{crit}$ be fixed and $\sigma_*,\sigma^*$ be as in Lemma \ref{Auslenkung14}.\\
{\bf(i)} If $\sigma < \sigma_*$, then there exists $\delta > 0$ such that
\begin{align*}
u_{in}^{\overline{\lambda}} (z) > u_{in}^{\lambda}(z) \,, \qquad 0 \leq \overline{\lambda} < \lambda < \delta \,, \quad z \in (-1,1)\,.
\end{align*} 
{\bf(ii)} If $\sigma > \sigma^*$, then there exist $\delta > 0$, $r_0 \in (0,1)$ and $n \in \mathbb{N}$ with $2/n < \min \lbrace r_0 , 1-r_0 \rbrace$ such that
\begin{align*}
 u_{in}^{\overline{\lambda}} (z) > u_{in}^{\lambda}(z) \,, \quad 0 \leq \overline{\lambda} < \lambda < \delta \,, \quad z \in  [-r_0 +1/n , r_0-1/n]
\end{align*}
as well as 
\begin{align*}
\qquad \quad  \ \ \, u_{in}^{\overline{\lambda}} (z) < u_{in}^{\lambda}(z) \,, \quad 0 \leq \overline{\lambda} < \lambda < \delta \,, \quad z \in  (-1,-r_0-1/n] \cup  [r_0+1/n, 1)\,.
\end{align*}
Moreover, $u^\lambda_{in}$ intersects $u^{\overline{\lambda}}_{in}$ on $(-1,1)$ in exactly two points $z_1, z_2$ with 
\begin{align*}
z_1 \in (-r_0-1/n,-r_0+1/n)\,, \qquad z_2 \in (r_0-1/n, r_0+1/n)\,,
\end{align*}
and $u_{in}^\lambda$ is strictly decreasing on $[-r_0-1/n, -r_0 + 1/n]$ as well as strictly increasing on $[r_0-1/n, r_0 + 1/n]$.
\end{thm}

\begin{proof}
 {\bf(i)} By Lemma \ref{Auslenkung14}\,(i), this follows exactly as in the proof of Theorem \ref{DeflectionOut}.\\
 {\bf(ii)} The argument is again quite similar to the one in Theorem \ref{DeflectionOut}: 
 First, we use Taylor's expansion as in Theorem \ref{DeflectionOut} together with Lemma \ref{Auslenkung14}\,(ii) to deduce the existence of $r_0 \in (0,1)$ and $n \in \mathbb{N}$ with $1/n$ small enough (which replaces $\varepsilon$ from the proof of Theorem \ref{DeflectionOut}) as well as $\delta > 0$ such that 
 \begin{align}
 &u^\lambda_{in}(z) < u^{\overline{\lambda}}_{in}(z) \,, \qquad\qquad z \in [-r_0+1/n,r_0-1/n]\,, \label{Auslenkung17}\\ 
 &u^\lambda_{in}(z) > u^{\overline{\lambda}}_{in}(z) \,, \qquad\qquad z \in [-1+1/n,-r_0-1/n] \cup [r_0+1/n,1-1/n]\,,\label{Auslenkung18}\\
 &\partial_z u^\lambda_{in}(z) > \partial_z u^{\overline{\lambda}}_{in}(z) \,, \qquad z \in  [-1,-1+1/n] \cup [r_0-1/n,r_0+1/n] \,,\label{Auslenkung19}\\
 &\partial_z u^\lambda_{in}(z) < \partial_z u^{\overline{\lambda}}_{in}(z) \,, \qquad z \in  [-r_0-1/n,-r_0+1/n] \cup [1-1/n,1]\,,  \label{Auslenkung20}
 \end{align}
for $0 \leq \overline{\lambda} < \lambda < \delta$. Next, we deduce from \eqref{Auslenkung18} - \eqref{Auslenkung20} and the fact that $u^\lambda_{in}$ as well as $u^{\overline{\lambda}}_{in}$ satisfy Dirichlet boundary conditions that 
\begin{align}
\qquad \quad  \ \ \, u_{in}^{\overline{\lambda}} (z) < u_{in}^{\lambda}(z) \,, \quad  z \in  (-1,-r_0-1/n] \cup  [r_0+1/n, 1)\,, \label{Auslenkung21}
\end{align}
for $0 \leq \overline{\lambda} < \lambda < \delta$. Moreover, since $$u^0_{in}(z)= \frac{\mathrm{cosh}(c_{in}z)}{\mathrm{cosh}(c_{in})} -1$$ with derivative 
\begin{align*}
\partial_z u^0_{in}(z)= \frac{\mathrm{sinh}(c_{in}z)}{\sigma}\quad \begin{cases} &\leq 0 \,, \qquad z \leq 0\,, \\
&>0 \,, \qquad z >0\,,\end{cases}
\end{align*}
we infer from \eqref{Auslenkung19} with $\overline{\lambda}=0$ that $u^\lambda_{in}$ is strictly increasing on $[r_0-1/n, r_0+1/n]$. Similarly, \eqref{Auslenkung20} yields that $u^\lambda_{in}$ is strictly decreasing on $ [-r_0-1/n,-r_0+1/n]$. 
It remains to study the intersection points of $u^\lambda_{in}$ and $u^{\overline{\lambda}}_{in}$. To this end, we deduce from \eqref{Auslenkung17} and \eqref{Auslenkung21} that $u^\lambda_{in}$ and $u^{\overline{\lambda}}_{in}$ may only intersect on
 \begin{align*}
(-r_0-1/n,-r_0+1/n)\cup (r_0-1/n, r_0+1/n) \subset (-1,1)\,.
\end{align*}
Thanks to \eqref{Auslenkung17} and \eqref{Auslenkung21}, we find
\begin{align*}
u^\lambda_{in}(-r_0-1/n) >  u^{\overline{\lambda}}_{in}(-r_0-1/n) \,, \qquad   u^\lambda_{in}(-r_0+1/n)  < u^{\overline{\lambda}}_{in}(-r_0+1/n)
\end{align*} 
for $0 \leq \overline{\lambda} < \lambda < \delta$. Consequently, \eqref{Auslenkung20} yields that $u^\lambda$ and $u^{\overline{\lambda}}$ have exactly one intersection point $z_1$ in $(-r_0-1/n, -r_0+1/n)$. 
Finally, note that the existence of the second intersection point $z_2$ in $(r_0-1/n, r_0+1/n)$ follows similarly.
\end{proof}

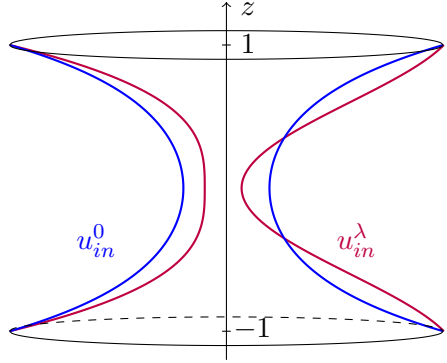
\begin{figure}[h]
\begin{tikzpicture}
\begin{axis}[samples=200, axis x line=none, axis y line=none, domain=-1:1]
\addplot[color=purple, thick]({0.5*sin(150*x-90)+0.57},{x});
\addplot[color=blue, thick]({cosh(2.3*x)/cosh(2.3)},{x});
\addplot[color=purple, thick]({-(x^2)^1.5*0.9-0.1},{x});
\addplot[color=blue, thick]({-cosh(2.3*x)/cosh(2.3)},{x});
\addplot[color=black]({x},{1-0.1*(1-x^2)^0.5});
\addplot[color=black]({x},{1+0.1*(1-x^2)^0.5});
\addplot[color=black]({x},{-1-0.1*(1-x^2)^0.5});
\addplot[color=black, dashed]({x},{-1+0.1*(1-x^2)^0.5});
\addplot[->, domain=-1.2:1.3]({0},{x});
\addplot[domain=-0.02:0.02]({x},{-1});
\addplot[domain=-0.02:0.02]({x},{1});
\node[above] at (112, 7) (p1) { $-1$ };
\node[above] at (110, 207) (p2) { $1$};
\node[above] at (110, 234) (p2) { $z$};
\node[above,blue] at (40, 65) (p5) { $u_{in}^0$};
\node[above,purple] at (160, 65) (p6) { $u^\lambda_{in}$};
\end{axis}
\end{tikzpicture}
\caption[Deflections Emanating from $u_{in}^0$ in the Small Aspect Ratio Model]{Qualitative behaviour of the deflection $u^\lambda_{in}$ (red) of the inner catenoid $u_{in}^0$ (blue) for small applied voltages in the small aspect ratio model \eqref{A.56}. On the left a possible deflection for $\sigma \in (\sigma_{crit},\sigma_*)$ is depicted, while on the right a possible deflection for $\sigma \in (\sigma^*,\infty)$ is shown. 
For $\sigma \in (\sigma_*,\sigma^*)$, the qualitative behaviour of the deflection is unknown. The cylinder is not depicted in this graphic.}\label{figurdeflection}
\end{figure}
\end{section}

\section*{Acknowledgement}
This paper contains results and edited text from my PhD-thesis. I express my gratitude to my PhD-supervisor Christoph Walker for his support.

\appendix 

\section{Anti-Maximum Principle}\label{AMP}

 Recall from \eqref{DefPhi} that
\begin{align*}
\varphi(z)= \mathrm{cosh}(c_{in}z) - c_{in} \, z \, \mathrm{sinh}(c_{in}z) 
\end{align*}
is the unique solution to 
\begin{align*}
\begin{cases}
&0 = - \partial_z \Big ( \displaystyle{\frac{1}{\mathrm{cosh}^2(c_{in}z)} \partial_z \varphi \Big ) - \frac{c_{in}^2}{\mathrm{cosh}^2(c_{in}z)}} \varphi \quad \text{on} \ (-1,1)\,,\\
&\varphi(0)= 1 \,, \quad \partial_z \varphi(0)=0
\end{cases} 
\end{align*}

with $c_{in}>0$. Moreover, recall that
\begin{align*}
-DF(u_{in}^0) v = -\sigma^2\, \bigg [ \partial_z \Big ( \frac{1}{\mathrm{cosh}^2(c_{in}z)} \,\partial_z v\Big ) + \frac{c_{in}^2}{\mathrm{cosh}^2(c_{in}z)} v \bigg ]\,, \qquad v \in W^2_{q,D}(-1,1)
\end{align*}
from \eqref{defB2}. We present a (slightly adapted) criterion from \cite[Theorem 2.3]{Shi04} to decide whether or not an anti-maximum principle applies to $-DF(u^0_{in})$:

\begin{lem}\label{Auslenkung13}$($\cite{Shi04}$)$\\
Let $f \in C\big ([-1,1]\big )$ with $f(z)=f(-z)>0$ for each $z \in [-1,1]$ and consider the even function $v:= \big (-DF(u_{in}^0) \big )^{-1} f \in W^2_{q,D}(-1,1) \cap C^2 \big ([-1,1] \big )$.  \\
{\bf(i)} If 
 \begin{align*}
  \int_{-1}^1  f(z)\, \varphi(z) \,\mathrm{d} z > 0\,,
\end{align*}
then $v < 0$ on $(-1,1)$ with $\partial_z v(-1) < 0$ and $\partial_zv(1) > 0$. \\
{\bf(ii)} If 
\begin{align*}
  \int_{-1}^1  f(z)\, \varphi(z) \,\mathrm{d} z < 0\,,
\end{align*}
then $v$ is sign-changing: there exists $r_0 \in (0,1)$ such that $v < 0$ on $(-r_0,r_0)$ and $v>0$ on $(-1,-r_0) \cup (r_0,1)$ as well as
\begin{align*}
 &\partial_z v(-1)  > 0 \,, \quad \partial_z v(-r_0) < 0 \,, \\
 &\partial_z v(r_0)>0 \,, \quad \partial_z v (1) <0\,.
\end{align*} 
\end{lem}

\begin{proof} Put $c=c_{in}$. Since $f$ is strictly positive and Corollary \ref{Stabcatcor}
ensures that $DF(u_{in}^0)$ has exactly one strictly positive eigenvalue, while all other eigenvalues are strictly negative, we can rely on \cite{Shi04}:\\
{\bf(i)} See \cite[Theorem 2.3]{Shi04}. \\
{\bf(ii)} In the proof of \cite[Theorem 2.3]{Shi04}, it is shown that the set
\begin{align*}
I_-:= \big \lbrace z \in (-1,1) \, \big \vert \, v(z)<0 \big \rbrace 
\end{align*}
coincides either with $(-1,1)$ or with $(-r_0,r_0)$ for some $0 <r_0 <1$. Because $v(\pm 1)=0$, $\partial_z v(0)=0$ due to symmetry, and $\partial_z\varphi(0)=0$ as initial data, integration by parts yields
\begin{align}
 &\quad \  \ \, \frac{1}{\mathrm{cosh}^2(c)} \varphi(1)\,\partial_zv(1)  \nonumber \\
 &= \Big [ \frac{1}{\mathrm{cosh}^2(cz)} \varphi(z) \,\partial_z v(z)-\frac{1}{\mathrm{cosh}^2(cz)} \partial_z\varphi(z) \,v(z) \Big ]_{z=0}^{z=1} \nonumber \\
 &= \int_0^1 \bigg(\partial_z \Big( \frac{1}{\mathrm{cosh}^2(cz)} \partial_z v(z) \Big) \varphi(z) - \partial_z \Big ( \frac{1}{\mathrm{cosh}^2(cz)} \partial_z\varphi(z) \Big) v(z)\bigg) \,\mathrm{d} z\,. \label{LagrangeIdentity0}
\end{align}
Adding $\pm \displaystyle\frac{c^2}{\mathrm{cosh}^2(cz)} v(z)\varphi(z)$ to \eqref{LagrangeIdentity0} and using the differential equation for $\varphi$ we see that
\begin{align}
 \quad \  \ \, \frac{1}{\mathrm{cosh}^2(c)} \varphi(1)\,\partial_z v(1) 
 &= \frac{1}{\sigma^2} \,\int_0^1 \big(DF(u_{in}^0)v\big )(z)\,\varphi(z) \, \mathrm{d} z \nonumber \\
 &= -\frac{1}{\sigma^2}\,\int_0^1  f(z) \varphi(z) \, \mathrm{d} z \nonumber \\
 &= - \frac{1}{2\sigma^2} \,\int_{-1}^1  f(z) \varphi(z) \, \mathrm{d} z> 0\,.  \label{LagrangeIdentitaet}
\end{align}
We deduce from \eqref{LagrangeIdentitaet} and $\varphi(1) < 0$ that $\partial_z v(1) < 0$. Since $v(\pm1)=0$, it follows that $v$ is non-negative close to $z=1$, and consequently $I_-\neq(-1,1)$. Hence, we have $I_-=(-r_0,r_0)$ for some $0 < r_0 <1$. We note that also $\partial_z v(-1)>0$ by symmetry. To show that $v$ is strictly positive on $(-1,-r_0) \cup (r_0,1)$, we assume for contradiction that $v(z_0)=0$ for some $z_0$ with $r_0 < \vert z_0 \vert <1$. 
Since $v \geq 0$ close to $z_0$, it follows that $v$ has a local minimum at $z_0$ and hence necessarily $\partial_zv(z_0)=0$. 
But then we find that
\begin{align}
0 &> -\frac{1}{\sigma^2}  f(z_0) \notag \\ 
&= \frac{1}{\sigma^2} \big [DF(u_{in}^0)v \big ](z_0)\notag \\
&= \partial_z \Big ( \frac{1}{\mathrm{cosh}^2(cz)} \Big ) \Big \vert_{z=z_0} \partial_zv(z_0) + \frac{1}{\mathrm{cosh}^2(cz_0)} \partial_{zz} v(z_0) \notag \\
&\ \ \ \ + \frac{c^2}{\mathrm{cosh^2}(cz_0)} v(z_0) \notag \\
&=\frac{1}{\mathrm{cosh}^2(cz_0)} \partial_{zz} v(z_0)\,, \label{Auslenkung13a}
\end{align}
i.e. $v$ has a strict local maximum at $z_0$ which is not possible.  
\end{proof}

\bibliographystyle{siam}
\bibliography{BibliographyDoc}
\end{document}